\documentclass{article}
\usepackage{amssymb}
\usepackage{amsmath}
\usepackage{amsthm}
\usepackage{caption}
\usepackage{graphicx}
\usepackage{pgfplots}
\usepackage{hyperref}
\usepackage{color} %%%temporary for adding comments
\usepackage{verbatim}   %%%%temporary for commenting out large sections
\usepackage{bm}
\theoremstyle{theorem}
\newtheorem{thm}{Theorem}
\newtheorem{lemma}[thm]{Lemma}
\newtheorem{corr}[thm]{Corollary}
\newtheorem*{conj}{Conjecture}
\newcommand{\N}{\mathbb{N}}

\author{Mary Grace Hanson and David A. Nash\thanks{Le Moyne College}}
\title{Minimal and maximal Numbrix puzzles}
\date{\today}
\begin{document}

\maketitle
\begin{abstract}
This paper explores special arrangements of clues in $m \times n$ Numbrix puzzles.  The maximum number of clues which fails to define an $m \times n$ puzzle is demonstrated for all $m$ and $n$.  In addition, a small upper bound on the minimum number of clues required to define an $m \times n$ puzzle is given for all $m$ and $n$ as well.  For small $m \geq 3$ our upper bound appears to actually give the minimum number and hence we conjecture that our bound may be sharp for all $m \geq 3$.
\end{abstract}

%%%%%%%%%%%%%%%%%
%\section{Introduction}
Numbrix puzzles were created by Marilyn vos Savant, author of a question and answer column in Parade Magazine called ``Ask Marilyn" and holder of the highest IQ ever recorded.  As part of her work for Parade, she also creates a daily Numbrix Puzzle that is posted to the Parade website. Savant typically creates $9\times 9$ puzzles with a varying number of clues, but for the purposes of this paper, the definition of Numbrix Puzzles has been expanded to any $m\times n$ board with $m \leq n \in \N$.  Our choice to force $m \leq n$ simply arises from the ability to reflect or rotate any $m \times n$ board into an $n \times m$ one.

A completed $m \times n$ Numbrix Puzzle is an $m \times n$ board containing the numbers $1$ through $mn$ so that consecutive numbers are side-to-side adjacent (either horizontally or vertically, not diagonally).  An incomplete Numbrix Puzzle has some number of squares on the board already filled in with chosen numerical ``clues" so that there is only one way to fill in the remaining numbers. It is the job of the solver to use the given clues together with the rules for Numbrix puzzles to deduce the rest of the entries.  Figure~\ref{fig:example} gives an example of a puzzle on a $3\times 3$ board for the reader to solve. %{\color{blue}(Maybe challenge the reader to solve it rather than demonstrating?)}

\begin{figure}[h!]
	\centering
	\begin{tikzpicture}
	\draw[gray] (0,0) -- (1.5,0);
	\draw[gray] (0,.5) -- (1.5,.5);
	\draw[gray] (0,1) -- (1.5,1);
	\draw[gray] (1.5,1.5) -- (0,1.5);
	\draw[gray] (0,0) -- (0,1.5);
	\draw[gray] (.5,0) -- (.5,1.5);
	\draw[gray] (1,0) -- (1,1.5);
	\draw[gray] (1.5,0) -- (1.5,1.5);
	\draw (.75,.25) 	node {$6$};
	\draw (1.25,.75) 	node {$2$};
	\end{tikzpicture}
	\caption{A $3\times 3$ Numbrix puzzle.}
	\label{fig:example}
\end{figure}

As we will see in Section~\ref{sec:Graph}, Numbrix puzzles have some nice connections to graph theory.  More specifically, the solutions to Numbrix puzzles correspond to directed Hamiltonian paths or circuits in rectangular grid graphs.  Counting Hamiltonian paths and circuits in general graphs is a difficult problem and the family of rectangular grid graphs is no exception.  Thompson \cite{Thompson} first raised this question and gave several preliminary results regarding Hamiltonian paths that start in one corner of the grid and end in another.  Since then, work has been done (see e.g.\ \cite{Enting}, \cite{CEG}, \cite{SS}, and \cite{CK}) to develop generating functions for counting various subsets of the Hamiltonian paths and circuits in various grid graphs.

On the more computational side, Itai, Papadimitriou, and Szwarcfiter \cite{IPS} have demonstrated that determining the existence of a Hamiltonian path with a chosen start and finish or the existence of a Hamiltonian circuit with a chosen start in general grid graphs are both NP-complete problems.  As for counting the number of such paths and circuits, in the square case ($m=n$), Mayer, Guez, and Dayantis \cite{MGD} developed a computer algorithm for counting all Hamiltonian paths, but were limited to counting up to $7 \times 7$ grids by computational time constraints.  More recently, Jacobsen \cite{Jacobsen} created an algorithm for counting the number of Hamiltonian chains, paths, and circuits in grid graphs in both 2 and 3 dimensions.  In the 2-dimensional setting and for square $n \times n$ grids specifically, his algorithm allowed him to count the number of chains up to $n=12$, the number of paths up to $n=17$, and the number of circuits up to $n=20$.  Many of these types of counts can be found in the On-line Encyclopedia of Integer Sequences (see e.g.\ A096969 or A143246). 

While we plan to make use of some of the related work in graph theory, this paper instead focuses on two distinctly different questions concerning Numbrix Puzzles: 1) On an $m\times n$ board what is the maximum number of clues that can be given without defining a Numbrix puzzle and 2) On an $m\times n$ board what is the minimum number of clues required to define a Numbrix puzzle?  Note that we will only call a collection of numbers on the board a set of ``clues" if there exists at least one possible solution containing those numbers in the given positions.  That way, we may directly avoid ``false clues'' with no solutions.

\section{Graph theory and observations}\label{sec:Graph}
As mentioned above, we can gain insight into our problem by drawing connections between Numbrix puzzles and graph theory.  We start by creating the underlying rectangular grid graph associated with an $m \times n$ board where the vertices correspond to the squares in our board and two vertices are connected by an edge whenever the associated squares are adjacent horizontally or vertically.  It is well-known that rectangular grid graphs are \emph{bipartite} graphs, meaning that they contain two disjoint subsets of the vertices such that there are no edges between any vertices in the same part.  This is equivalent to saying that the graph is \emph{two-colorable}, meaning that we need only two colors to color all vertices in the graph without any adjacent vertices having the same color.  For example, we may color our $m \times n$ grid graph in a checkerboard pattern starting with white in the upper left corner (see Figure~\ref{fig:checkerboard}).

\begin{figure}[h!]
	\centering
	\begin{tikzpicture}[scale = .5]
	\draw[gray] (0,3) rectangle (1,2);
	\draw[gray] (2,3) rectangle (3,2);
	\draw[gray] (4,3) rectangle (5,2);
	\draw[gray] (1,2) rectangle (2,1);
	\draw[gray] (3,2) rectangle (4,1);
	\draw[gray] (0,1) rectangle (1,0);
	\draw[gray] (2,1) rectangle (3,0);
	\draw[gray] (4,1) rectangle (5,0);
	\draw[gray,fill=black] (1,3) rectangle (2,2);
	\draw[gray,fill=black] (3,3) rectangle (4,2);
	\draw[gray,fill=black] (0,2) rectangle (1,1);
	\draw[gray,fill=black] (2,2) rectangle (3,1);
	\draw[gray,fill=black] (4,2) rectangle (5,1);
	\draw[gray,fill=black] (1,1) rectangle (2,0);
	\draw[gray,fill=black] (3,1) rectangle (4,0);
	\end{tikzpicture}
	\caption{A checkerboard coloring of a $3 \times 5$ board.}
	\label{fig:checkerboard}
\end{figure}
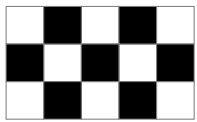

\newpage
Since the only edges in our graph must connect squares of different colors, this implies that, for any solution to a Numbrix puzzle, the odd and even numbers must alternate colors.  Notice that when $m$ and $n$ are both odd, as in the above example, then the odd part of the graph must have one more vertex than the even part.  Thus, when $m$ and $n$ are both odd, the odd numbers in a solution must appear in the ``white" part of our coloring, while the even numbers must appear in the ``black" part.

Beyond this coloring idea, the solution to any Numbrix puzzle can be used to define a \emph{directed Hamiltonian path} -- that is, a directed path which visits every vertex in the graph exactly once with no repeated edges -- by moving through the consecutive squares starting with one.  For our purposes, we'll drop the direction, so that each \emph{path} that we consider corresponds to at least two solutions -- one for each potential direction.  Thus, given any solution to a puzzle, there exists another solution in which each entry $1+j$ is replaced by $mn-j$ for each $0\leq j \leq mn-1$.  This idea turns out to be helpful, so we refer to it as the \emph{reversal property} of Numbrix puzzles.

\begin{lemma}[Reversal Property]
Given a solution to an $m \times n$ Numbrix puzzle, there exists another puzzle whose solution simply replaces each entry $1+j$ with $mn-j$ for each $0 \leq j \leq mn-1$.
\end{lemma}

In the special case when $m \geq 2$ and the solution to a puzzle also has 1 and $mn$ adjacent, then the underlying path has start and end next to each other.  Hence, by following the edge between them, such a solution defines a Hamiltonian circuit in the grid which we will refer to as a \emph{circular path}.  Given a circular path in a grid, we choose to forget the starting point in order to use the same circular path to define $2mn$ different solutions -- by choosing a start (or any particular value) and a direction.  The following Lemma regarding the existence of circular paths was known to Thompson \cite{Thompson}, however he left the proof to the reader, so we offer a short one here.

\begin{lemma}[Thompson]\label{lem:circularpaths}
With $m \geq 2$, an $m \times n$ board contains a circular path if and only if $m$ is even or $n$ is even
\end{lemma}	
\begin{proof}
$(\implies)$ If a circular path exists, then it places $1$ and $mn$ in adjacent squares.  Thus, $mn$ is an even number -- since adjacent squares must have opposite parity -- which implies that $m$ is even or $n$ is even.\\
$(\impliedby)$ Suppose without loss of generality that $m$ is even.  We then give an algorithm for creating a circular path on the board: Starting from the top left corner, take the path clockwise around the outer edge of the board, to the top right corner, then the bottom right corner, and finally the bottom left corner.  If $m=2$ we are done, otherwise this %``U" shaped 
path leaves an $(m-2)\times (n-1)$ board that still must be covered.  If we continue the path by zig-zagging back and forth along each row, the fact that $(m-2)$ is even, implies that this zigzag must end at the top left corner, adjacent to our original starting position.  Hence, as all squares have been traversed, it follows that this defines a circular path.  See Figure~\ref{fig:circularalgorithm} for an example of this process.
\end{proof}
	
	\begin{figure}[h!]
		\centering
		\begin{tikzpicture}[scale = .8]
		%horizontal
		\draw[gray] (0,0) -- (3.5,0);
		\draw[gray] (0,.5) -- (3.5,.5);
		\draw[gray] (0,1) -- (3.5,1);
		\draw[gray] (0,1.5) -- (3.5,1.5);
		\draw[gray] (0,2) -- (3.5,2);
		\draw[gray] (0,2.5) -- (3.5,2.5);
		\draw[gray] (0,3) -- (3.5,3);
		%vertical
		\draw[gray] (0,0) -- (0,3);
		\draw[gray] (.5,0) -- (.5,3);
		\draw[gray] (1,0) -- (1,3);
		\draw[gray] (1.5,0) -- (1.5,3);
		\draw[gray] (2,0) -- (2,3);
		\draw[gray] (2.5,0) -- (2.5,3);
		\draw[gray] (3,0) -- (3,3);
		\draw[gray] (3.5,0) -- (3.5,3);
		\draw[very thick] (.25,2.75) -- (3.25,2.75) -- (3.25,0.25) -- (0.25, 0.25);
		\draw[very thick] (0.25,0.25) -- (0.25,0.75) -- (2.75,0.75) -- (2.75,1.25) -- (0.25,1.25) -- (0.25,1.75) -- (2.75,1.75) -- (2.75,2.25) -- (0.25,2.25) -- (0.25,2.75);
		\end{tikzpicture}
		\caption{A circular path created when $m=6$ and $n=7$.}
		\label{fig:circularalgorithm}
	\end{figure}

The existence of a circular path on a board tells us immediately that one clue will not define a Numbrix puzzle.
 
\begin{corr}\label{cor:oneclue}
With $m \geq 2$, one clue does not define a puzzle on an $m\times n$ board, when $m$ is even or $n$ is even.
\end{corr}
\begin{proof}
From Lemma~\ref{lem:circularpaths} we know that there exists at least one circular path on any such board.  Let $x$ be a clue on the board, then regardless where $x$ is placed, we can use the path to complete the puzzle in two different ways by simply giving the path a direction. Therefore, one clue is not enough to define a unique solution.
\end{proof}

%\subsection{Zig-zag solutions}
As we move towards our two main questions, it will be helpful to keep a particular type of puzzle solution in mind.  We refer to any solution to an $m \times n$ puzzle whose entries  proceed back and forth across the rows starting from any corner as a ({horizontal) \emph{zig-zag solution}, see Figure~\ref{fig:zig-zag}.
	\begin{figure}[h!]
		\centering
		\begin{tikzpicture}[scale = .5]
		%horizontal
		\draw[gray] (0,0) rectangle (4,1);
		\draw[gray] (0,2) rectangle (4,3);
		\draw[gray] (0,0) rectangle (1,3);
		\draw[gray] (2,0) rectangle (3,3);
		\draw[gray] (4,1) -- (4,2);
		\draw (0.5,2.5) node {4};
		\draw (1.5,2.5) node {3};
		\draw (2.5,2.5) node {2};
		\draw (3.5,2.5) node {1};
		\draw (3.5,1.5) node {8};
		\draw (2.5,1.5) node {7};
		\draw (1.5,1.5) node {6};
		\draw (0.5,1.5) node {5};
		\draw (0.5,.5) node {12};
		\draw (1.5,.5) node {11};
		\draw (2.5,.5) node {10};
		\draw (3.5,.5) node {9};
		\end{tikzpicture}
		\caption{The $3 \times 4$ zig-zag solution starting in the top right corner.}
		\label{fig:zig-zag}
	\end{figure}

These solutions are interesting because they have a very distinctive structure that can be exploited readily when giving clues.   The important observation to make is that each entry in the first column has exactly one more adjacent number that has yet to be filled in and exactly one adjacent square (in column two) in which to place it.  Hence, filling in those adjacent numbers will completely fill column two.  For example, if we were given 4, 5, and 12 as in Figure~\ref{fig:zig-zag}, then we would immediately be able to fill in 3, 6, and 11 in the second column.  These new entries are now in the same predicament, each having one adjacent number and only one adjacent square to place it in.  Continuing this process column by column, the only way to complete the board is with exactly the zig-zag solution.  The same idea applies to any particular zig-zag solution.  Since we plan to make heavy use of this idea in what follows we pull it out for later reference.

\begin{lemma}[Zig-Zag Clues]\label{lem:zig-zag}
The entries in the first column of any $m \times n$ zig-zag solution define a Numbrix puzzle when given as clues.
\end{lemma}

%%%%%%%%%%%%%%%%%%%%%%%
\section{$1\times n$ and $2\times n$ boards}
The $1 \times n$ boards are too small for us to see the structure that appears in more general situations, so we choose to deal with them separately first.  Certainly, a $1 \times 1$ Numbrix puzzle isn't really a puzzle as there is only one way to complete it and there are, therefore, no clues required.  In the slightly more interesting case when $n \geq 2$, a $1 \times n$ puzzle can only be completed in two possible ways since there is only one undirected path in the underlying grid.  The solution depends only on which end has 1 and which has $n$ (or which direction we give to the path).  With this in mind, one can see that almost any single clue will completely determine the puzzle.  The one case where this is not true is when $n$ is odd and the clue you are given is $\frac{n+1}{2}$ exactly in the middle of the board, see Figure~\ref{fig:1xnMax}.  With these observations in mind, we can answer all of our questions for $1 \times n$ puzzles.

\begin{figure} [h!]
		\centering
		\begin{tikzpicture}[scale=0.6]
		\draw[gray] (0,0) -- (7,0); %bottom line%
		\draw[gray] (0,1) -- (7,1); %middle horizontal%
		\draw[gray] (0,0) -- (0,1); %left line%
		\draw[gray] (1,0) -- (1,1);%middle vertical%
		\draw[gray,dotted] (1.1,0.5) -- (2.9,0.5);
		\draw[gray] (3,0) -- (3,1) -- (4,1) -- (4,0) -- (3,0);
		\draw[gray,dotted] (4.1,0.5) -- (5.9,0.5);
		\draw[gray] (6,0) -- (6,1);%middle vertical%
		\draw[gray] (7,0) -- (7,1);%right line%
		\draw (3.5,.5) node {$\frac{n+1}{2}$};
		\end{tikzpicture}
		\caption{An unsolvable $1 \times n$ puzzle for $n$ odd.}
		\label{fig:1xnMax}
	\end{figure}
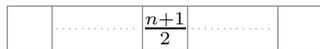

\begin{thm}[$1\times n$ Puzzles]\label{thm:1x1}~\\
(i) The minimum number of clues that will define a $1 \times n$ Numbrix puzzle is 0, if $n=1$, and is 1 otherwise.\\
(ii) With $n>1$, the maximum number of clues that can be given without defining a $1 \times n$ Numbrix puzzle is 1, if $n$ is odd, and is 0 otherwise.
\end{thm}

Next we can turn our attention to the $2\times n$ case ($n \geq 2$) which is a little bit more interesting as we have more room to work with.  Now, there are numerous possible paths through the underlying grid.  More importantly, thanks to Lemma~\ref{lem:circularpaths} we know that there exists a circular path in this board and that, therefore, no single clue can ever determine a puzzle.  Knowing that one clue will never suffice and in view of Lemma~\ref{lem:zig-zag} we can quickly see that two clues will be minimal as we may use the first column of any $2 \times n$ zig-zag solution as clues.

As for the maximum number of clues that fail to define a puzzle, we gain insight by considering the $2 \times 2$ case.  Observe that if we give clues (in bold) in opposite corners, then there will always be two ways to complete the board, see Figure~\ref{fig:2x2}.  In a similar way, we can fill in all but two boxes in any $2 \times n$ board and still be left with no way to determine the final two squares.

	\begin{figure} [h!]
		\centering
		\begin{tikzpicture}
		\draw[gray] (0,0) -- (1,0); %bottom line%
		\draw[gray] (0,.5) -- (1,.5); %middle horizontal%
		\draw[gray] (0,1) -- (1,1); %top line%
		\draw[gray] (0,0) -- (0,1); %left line%
		\draw[gray] (.5,0) -- (.5,1);%middle vertical%
		\draw[gray] (1,0) -- (1,1);%right line%
		\draw (.25,.25) node {\large \textbf{2}};
		\draw (.75,.75) node {\large \textbf{4}};
		\draw (.25,.75) node {\small 1};
		\draw (.75,.25) node {\small 3};
		\end{tikzpicture}
		\qquad
		\begin{tikzpicture}
		\draw[gray] (0,0) -- (1,0); %bottom line%
		\draw[gray] (0,.5) -- (1,.5); %middle horizontal%
		\draw[gray] (0,1) -- (1,1); %top line%
		\draw[gray] (0,0) -- (0,1); %left line%
		\draw[gray] (.5,0) -- (.5,1);%middle vertical%
		\draw[gray] (1,0) -- (1,1);%right line%
		\draw (.25,.25) node {\large \textbf{2}};
		\draw (.75,.75) node {\large \textbf{4}};
		\draw (.75,.25) node {\small 1};
		\draw (.25,.75) node {\small 3};
		\end{tikzpicture}
		\caption{Two possible solutions for a $2\times 2$ board with 2 clues.}
		\label{fig:2x2}
	\end{figure}

%{\color{blue}(I'm now thinking that we might want to change the wording of ``maximum'' stuff...)}
\begin{thm}[$2\times n$ Puzzles]\label{thm:2xn}~\\
(i) The minimum number of clues that will define a $2 \times n$ Numbrix puzzle is 2.\\
(ii) The maximum number of clues that will not define a $2 \times n$ Numbrix puzzle is $2n-2$.
\end{thm}
\begin{proof} (i) The following board defines a Numbrix puzzle by Lemma~\ref{lem:zig-zag}.
$$
		\begin{tikzpicture}
		\draw[gray] (0,0) -- (3,0); %bottom line%
		\draw[gray] (0,.5) -- (3,.5); %middle horizontal%
		\draw[gray] (0,1) -- (3,1); %top line%
		\draw[gray] (0,0) -- (0,1); %left line%
		\draw[gray] (.5,0) -- (.5,1);%middle vertical%
		\draw[gray] (1,0) -- (1,1);%right line%
		\draw[gray,dotted] (1.1,0.25) -- (2.4,0.25);
		\draw[gray,dotted] (1.1,0.75) -- (2.4,0.75);
		\draw[gray] (2.5,0) -- (2.5,1);%middle vertical%
		\draw[gray] (3,0) -- (3,1);%right line%
		\draw (.25,.25) node {$\bm{2n}$};
		\draw (.25,.75) node {\textbf{1}};
		\end{tikzpicture}
$$

(ii) Observe that if we start with the two clues that we gave in the $2 \times 2$ setting, then we can complete the rest of the board anyway we would like.  Filling in all of those entries will leave us with a board with $2n-2$ clues that has no unique solution.  For example, if we choose one possible (vertical) ``zigzag" completion then we would have the board:
$$
\begin{tikzpicture}
		\draw[gray] (0,0) -- (5,0); %bottom line%
		\draw[gray] (0,.5) -- (5,.5); %middle horizontal%
		\draw[gray] (0,1) -- (5,1); %top line%
		\draw[gray] (0,0) -- (0,1); %left line%
		\draw[gray] (.5,0) -- (.5,1);%middle vertical%
		\draw[gray] (1,0) -- (1,1);%right line%
		\draw[gray] (1.5,0) -- (1.5,1);
		\draw[gray] (2,0) -- (2,1);
		\draw[gray,dotted] (2.1,0.25) -- (2.9,0.25);
		\draw[gray,dotted] (2.1,0.75) -- (2.9,0.75);
		\draw[gray] (3,0) -- (3,1); %left line%
		\draw[gray] (4,0) -- (4,1);%middle vertical%
		\draw[gray] (5,0) -- (5,1);%right line%
		\draw (.25,.25) node {\textbf{2}};
		\draw (.75,.75) node {\textbf{4}};
		\draw (1.25,0.75) node {\textbf{5}};
		\draw (1.75,0.75) node {\textbf{8}};
		\draw (1.25,0.25) node {\textbf{6}};
		\draw (1.75,0.25) node {\textbf{7}};
		\draw (3.5,0.75) node {\textbf{2n-2}};
		\draw (3.5,0.25) node {\textbf{2n-3}};
		\draw (4.5,0.75) node {\textbf{2n-1}};
		\draw (4.5,0.25) node {\textbf{2n}};
\end{tikzpicture}
$$
Note that while we have the clue $2n$ in the bottom row, it will actually end up in the top row if the number of columns to be filled, $n-2$, is even.  Certainly, having any board filled with $2n-1$ clues will define a Numbrix puzzle as there will be only one empty square and only one missing number to put in that position.  Hence, we may conclude that $2n-2$ clues is maximal.
\end{proof}

\section{The more general setting}
From here on, we will be looking only at $m \times n$ boards with $3 \leq m \leq n$.  The best place to start is with a generalization of Theorem~\ref{thm:2xn} part (ii).  In general, for any $m \times n$ board there exists a choice of $mn-2$ clues which will not define a Numbrix puzzle.  We demonstrate one such choice below:

\begin{thm}[Maximum Clues]\label{thm:maxclues}~\\
For an $m \times n$ board with $3 \leq m \leq n$, the maximum number of clues possible without defining a puzzle is $mn-2$.
\end{thm}
\begin{proof}
We now generalize the $2 \times n$ case to $m \times n$ with $m \geq 3$.  First consider the top two rows.  If $n$ is odd, then the clues given in the proof of Theorem~\ref{thm:2xn} will have $2n$ will appear in the last entry of row two.  If instead $n$ is even, then we can reflect these clues across the boundary line so that again $2n$ will appear at the end of row two.  Either way, we are now left to fill the other $(m-2) \times n$ board in any way we would like (filling in all clues) starting from the top right corner.  To demonstrate one possible example, we could zig-zag across the rows.  %We must demonstrate that there is at least one way to do so.  As an example, we could create a zig-zag solution for the $(m-2)\times n$ board starting in the top right and then shift all of the entries by $2n$ and place it below our first two rows.

Since the $2 \times 2$ block in the upper left corner still cannot be determined, we have successfully given $mn-2$ clues without defining a puzzle.  As in the $2 \times n$ case, giving $mn-1$ clues definitely must define a puzzle as there is only one empty square and only one number left to place in it.  %Thus, $mn-2$ clues is maximal.
\end{proof}

Determining the minimum number of clues necessary to define an $m \times n$ puzzle is a much more difficult task.  In fact, in the general setting we plan only to demonstrate a nice upper bound on the number.  After explaining the upper bound we will then comment directly on small cases in which we believe we can say more.

\begin{thm} For an $m \times n$ grid with $3 \leq m \leq n$, the minimum number of clues required to define a Numbrix puzzle is less than or equal to $\lceil\frac{m}{2}\rceil$.
\end{thm}
\begin{proof}
Our main argument starts when $m$ is at least 4, so we consider the case when $m=3$ separately.  Consider the following $3 \times n$ board with two clues:
$$
\begin{tikzpicture}
		\draw[gray] (-1,0) -- (3,0); %bottom line%
		\draw[gray] (-1,.5) -- (3,.5); %middle horizontal%
		\draw[gray] (-1,1) -- (3,1); %top line%
		\draw[gray] (-1,-.5) -- (3,-.5);
		\draw[gray] (-1,-.5) -- (-1,1); %left line%
		\draw[gray] (.5,-.5) -- (.5,1);%middle vertical%
		\draw[gray] (1,-.5) -- (1,1);%right line%
		\draw[gray,dotted] (1.1,0.75) -- (2.4,0.75);
		\draw[gray,dotted] (1.1,0.25) -- (2.4,0.25);
		\draw[gray,dotted] (1.1,-0.25) -- (2.4,-0.25);
		\draw[gray] (2.5,-.5) -- (2.5,1);%middle vertical%
		\draw[gray] (3,-.5) -- (3,1);%right line%
		\draw (-.25,.25) node {$\bm{n+1}$};
		\draw (-.25,-.25) node {$\bm{3n}$};
\end{tikzpicture}
$$
The important observation to make here is that $n+2$ cannot appear above $n+1$ as the path from $n+1$ to $3n$ would then create a loop that makes it impossible to complete the puzzle.  It follows that $n$ must appear above $n+1$, giving us the entire first column of a $3 \times n$ zig-zag solution.  Now, we may apply Lemma~\ref{lem:zig-zag}.

Next, we turn our attention to the $m\geq 4$ situation.  Using the Division Algorithm we may write $m=4q+r$ for some integer $q\geq 1$ and some $r=0$, 1, 2, or 3.  Our proscribed clues come most naturally in $4 \times n$ blocks with two clues in each block.  For any $1 \leq i \leq q$ we'll call the $i$th block the one contained in rows $4i-3$, $4i-2$, $4i-1$, and $4i$.  The entries in each block are the same up to a shift by the number of squares in all of the blocks above it.  For simplicity we'll denote the shift for the $i$th block, $4n(i-1)$, by $s_i$, see Figure~\ref{fig:4xnBlock}.

\begin{figure} [h!]
		\centering
\begin{tikzpicture}[scale=.5]
		\draw[gray] (0,0) rectangle (10,1);
		\draw[gray] (0,2) rectangle (10,3);
		\draw[gray] (0,3) rectangle (10,4);
		\draw[gray] (0,0) rectangle (4,4);
		\draw[gray] (9,0) rectangle (10,4);
		\draw[gray,dotted] (4.1,0.5) -- (8.9,0.5);
		\draw[gray,dotted] (4.1,1.5) -- (8.9,1.5);
		\draw[gray,dotted] (4.1,2.5) -- (8.9,2.5);
		\draw[gray,dotted] (4.1,3.5) -- (8.9,3.5);
		\draw (2,1.5) node {\small $\bm{s_i + 3n}$};
		\draw (2,2.5) node {\small $\bm{s_i + (n+1)}$};
		\draw (12,.5) node {\small Row $4i$};
		\draw (12,1.5) node {\small Row $4i-1$};
		\draw (12,2.5) node {\small Row $4i-2$};
		\draw (12,3.5) node {\small Row $4i-3$};
\end{tikzpicture}
\caption{The $i$th $4 \times n$ block in an $m \times n$ Numbrix puzzle.}
\label{fig:4xnBlock}	
\end{figure}

Just as in the $3 \times n$ situation, our clues imply that $s_i + (n+2)$ cannot be placed above $s_i+(n+1)$ and, similarly, $s_i+(3n-1)$ cannot be placed below $s_i+3n$ as, in either case, the path connecting $s_i + (n+1)$ to $s_i+3n$ would then make it impossible to complete the puzzle.  Hence, it follows that $s_i + n$ must appear above $s_i+(n+1)$ and $s_i+(3n+1)$ must appear below $s_i+3n$.  Thus, the clues in each $4 \times n$ block serve to immediately imply entries in the entire first column of the block.  

Now we describe how to place clues in the final $r$ rows when $r=1$, 2, or 3.  Remember that to remain under our bound of $\lceil\frac{m}{2}\rceil$ we must use no more than $\lceil \frac{r}{2} \rceil$ clues. Our goal in each case is to pick clues that will force the rest of the first column to be completely filled.  If $r=1$ or 2, we will place one more clue $s_q+5n$ in the first column of row $4q+1$.  When $r=2$, the path connecting $s_q+(3n+1)$ (the last implied entry in the $q$th block) to $s_q+5n$ again implies that the entry $s_q + (5n+1)$ must appear directly below $s_q +5n$.  Finally, if $r=3$ we place $s_{q+1}+(n+1)$ and $s_{q+1}+3n$ in the first column of the second to last and last rows respectively.  The path that must connect these two clues implies that $s_{q+1}+n$ must appear above $s_{q+1}+(n+1)$.  In all cases our clues have determined all of the entries in the first column of the board.  Moreover, by design, these entries are exactly the first column of an $m \times n$ zig-zag solution which starts in the top right corner, hence, by Lemma~\ref{lem:zig-zag} we know that the only solution is that zig-zag solution.  %Hence, it is possible for $\lceil \frac{m}{2} \rceil$ clues to define an $m \times n$ puzzle.
\end{proof}

Observe that for the $4 \times n$ boards, we can conclude that our upper bound ($\lceil \frac{4}{2} \rceil = 2$) must actually be minimal as there exists a circular path for each $n$ making it impossible for one clue to suffice by Corollary~\ref{cor:oneclue}.

\begin{corr}[Minimum for $4 \times n$]\label{cor:4xn}
The minimum number of clues that will define a $4 \times n$ Numbrix puzzle is 2.
\end{corr}

%%%%%%%%%%%%%%%%%%%%%%
\subsection{$3 \times n$ boards}
For $3 \times n$ boards, our upper bound strongly suggests that the minimum number of clues is again two.  The reader might be wondering why it doesn't immediately prove that two clues is minimal, seeing as one clue was insufficient for $2 \times n$ boards.  Certainly it seems obvious that one clue should be insufficient for $3 \times n$ boards (and anything larger) as well.  The problem here is two-fold: First, in the $2 \times n$ setting, there was a circular path to rely on, but no such path will exist in a $3 \times n$ board with $n$ odd.  And more generally, without a nice way to induct from smaller boards to larger ones (since there is no guarantee that the solution in an $m \times n$ puzzle will contain within it a solution to a smaller puzzle), we cannot even definitively state that the minimum number of clues must weakly increase as boards get larger (although we believe that it must).  In the theorem that follows, we will give a constructive proof that one clue is insufficient after using symmetry to make some reductions.

\begin{thm}[Minimum for $3 \times n$]\label{thm:3xn}~\\
The minimum number of clues that will define a $3 \times n$ Numbrix puzzle is 2.
\end{thm}
\begin{proof} Based on our above work and the fact that Corollary~\ref{cor:oneclue} will apply whenever $n$ is even, it suffices to show that one clue is not enough for any $3 \times n$ board with $n$ odd.  Let $x$ denote the numerical part of the clue given.  In the case when $n=3$, observe that no matter where $x$ appears, it will be placed on an axis of reflective symmetry (either horizontal, vertical, or a diagonal).  Hence, by reflecting any solution containing that clue we obtain a different solution that still contains it.

In a similar way, when $n \geq 5$, if $x$ appears in the second row or in column $\frac{n+1}{2}$, then by reflecting one solution we will obtain another which contains the same clue.  Therefore, up to symmetry, we need only consider the situation in which $x$ appears in the first row and in column $j$ for some $1 \leq j < \frac{n+1}{2}$.  For simplicity, we will refer to these $\frac{n-1}{2}$ squares as \emph{positions} 1 through $\frac{n-1}{2}$.  In addition, using the reversal property, we only need to consider cases with $1 \leq x \leq \frac{3n+1}{2}$.  Finally, recall that our graph coloring implies that the parity of each entry must match the parity of its position.  We are now ready to construct two solutions for each possible clue given these restrictions.

Let $J$ denote the largest odd number such that $J < \frac{n+1}{2}$.  Certainly if $x=1$ and is placed in position $J$, then we may draw the path to the right end of row one and then zig-zag vertically along the columns back to column $J$.  Since $J$ is odd, this zig-zagging will end in the third row and hence, from here, we have a $3 \times (J-1)$ board to fill starting in the bottom right corner.  We can finish by zig-zagging along the rows or along the columns, hence there are at least two solutions, see Figure~\ref{fig:x=1soln}.

	\begin{figure}[h!]
		\centering
		\begin{tikzpicture}[scale = .8]
		%horizontal
		\draw[gray] (0,0) -- (3.5,0);
		\draw[gray] (0,.5) -- (3.5,.5);
		\draw[gray] (0,1) -- (3.5,1);
		\draw[gray] (0,1.5) -- (3.5,1.5);
		%vertical
		\draw[gray] (0,0) -- (0,1.5);
		\draw[gray] (.5,0) -- (.5,1.5);
		\draw[gray] (1,0) -- (1,1.5);
		\draw[gray] (1.5,0) -- (1.5,1.5);
		\draw[gray] (2,0) -- (2,1.5);
		\draw[gray] (2.5,0) -- (2.5,1.5);
		\draw[gray] (3,0) -- (3,1.5);
		\draw[gray] (3.5,0) -- (3.5,1.5);
		\draw[fill=black] (1.25,1.25) circle [radius=0.1];
		\draw[very thick,->] (1.25,1.25) -- (2.25,1.25);
		\draw[very thick] (1.25,1.25) -- (3.25,1.25) -- (3.25,0.25) -- (2.75,0.25) -- (2.75, 0.75) -- (2.25,0.75) -- (2.25,0.25) -- (1.75,0.25) -- (1.75,0.75) -- (1.25,0.75) -- (1.25,0.25);
		\draw[very thick] (1.25,0.25) -- (0.25,0.25) -- (0.25,0.75) -- (0.75,0.75) -- (0.75,1.25) -- (0.25,1.25);
		\end{tikzpicture}
		\qquad
		\begin{tikzpicture}[scale = .8]
		%horizontal
		\draw[gray] (0,0) -- (3.5,0);
		\draw[gray] (0,.5) -- (3.5,.5);
		\draw[gray] (0,1) -- (3.5,1);
		\draw[gray] (0,1.5) -- (3.5,1.5);
		%vertical
		\draw[gray] (0,0) -- (0,1.5);
		\draw[gray] (.5,0) -- (.5,1.5);
		\draw[gray] (1,0) -- (1,1.5);
		\draw[gray] (1.5,0) -- (1.5,1.5);
		\draw[gray] (2,0) -- (2,1.5);
		\draw[gray] (2.5,0) -- (2.5,1.5);
		\draw[gray] (3,0) -- (3,1.5);
		\draw[gray] (3.5,0) -- (3.5,1.5);
		\draw[fill=black] (1.25,1.25) circle [radius=0.1];
		\draw[very thick,->] (1.25,1.25) -- (2.25,1.25);
		\draw[very thick] (1.25,1.25) -- (3.25,1.25) -- (3.25,0.25) -- (2.75,0.25) -- (2.75, 0.75) -- (2.25,0.75) -- (2.25,0.25) -- (1.75,0.25) -- (1.75,0.75) -- (1.25,0.75) -- (1.25,0.25);
		\draw[very thick] (1.25,0.25) -- (0.75,0.25) -- (0.75,1.25) -- (0.25,1.25) -- (0.25,0.25);
		\end{tikzpicture}
		\caption{Two possible paths starting in position $J=3$ on a $3 \times 7$ board.}
		\label{fig:x=1soln}
	\end{figure}
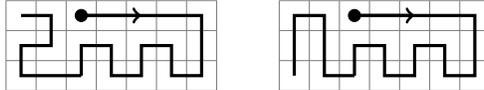

 If $3 \leq x < J$ and odd we can play the same game by pulling the start of our path back to the left, i.e.\ placing 1, 2, 3, etc.\ consecutively along the top row (starting in position $J-x+1$) so that $x$ appears in position $J$. Here we can mimic the construction from above as well, doing our vertical zig-zag to column $J-x+1$ (which will again stop in the third row since $J-x+1$ is odd) and then having at least two ways to complete the $3 \times (J-x)$ board that remains.
 
 For $x \geq J$ and odd, we will pull the start of the path back to position 1 and then (for $x > J$) will continue to pull back by zig-zagging vertically in the columns.  Observe that each column used increases $x$ by 2, thereby continuing to hit consecutive odd numbers in position $J$.  In each of these cases, we will take the front end of the path to the right end of the first row as before, but after this we will be left with an $2 \times l$ board to fill where $l=n-\frac{x-J}{2}$, since $\frac{x-J}{2}$ is the number of columns needed in the path to put $x$ in position $J$.  As long as $l \geq 2$ there will be at least two ways to complete the board (zig-zagging along the rows or columns), see Figure~\ref{fig:x>Jsoln}.
 
 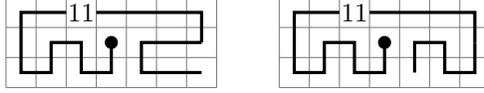
\begin{figure}[h!]
		\centering
		\begin{tikzpicture}[scale = .8]
		%horizontal
		\draw[gray] (0,0) -- (3.5,0);
		\draw[gray] (0,.5) -- (3.5,.5);
		\draw[gray] (0,1) -- (3.5,1);
		\draw[gray] (0,1.5) -- (3.5,1.5);
		%vertical
		\draw[gray] (0,0) -- (0,1.5);
		\draw[gray] (.5,0) -- (.5,1.5);
		\draw[gray] (1,0) -- (1,1.5);
		\draw[gray] (1.5,0) -- (1.5,1.5);
		\draw[gray] (2,0) -- (2,1.5);
		\draw[gray] (2.5,0) -- (2.5,1.5);
		\draw[gray] (3,0) -- (3,1.5);
		\draw[gray] (3.5,0) -- (3.5,1.5);
		\draw[fill=black] (1.75,0.75) circle [radius=0.1];
		\draw (1.25,1.25) node {11};
		\draw[very thick] (1.75,0.75) -- (1.75,0.25) -- (1.25,0.25) -- (1.25,0.75) -- (0.75,0.75) -- (0.75,0.25) -- (0.25,0.25) -- (0.25,1.25) -- (1,1.25);
		\draw[very thick] (1.5,1.25) -- (3.25,1.25) -- (3.25,0.75);
		\draw[very thick] (3.25,0.75) -- (2.25,0.75) -- (2.25,0.25) -- (3.25,0.25);
		\end{tikzpicture}
		\qquad
		\begin{tikzpicture}[scale = .8]
		%horizontal
		\draw[gray] (0,0) -- (3.5,0);
		\draw[gray] (0,.5) -- (3.5,.5);
		\draw[gray] (0,1) -- (3.5,1);
		\draw[gray] (0,1.5) -- (3.5,1.5);
		%vertical
		\draw[gray] (0,0) -- (0,1.5);
		\draw[gray] (.5,0) -- (.5,1.5);
		\draw[gray] (1,0) -- (1,1.5);
		\draw[gray] (1.5,0) -- (1.5,1.5);
		\draw[gray] (2,0) -- (2,1.5);
		\draw[gray] (2.5,0) -- (2.5,1.5);
		\draw[gray] (3,0) -- (3,1.5);
		\draw[gray] (3.5,0) -- (3.5,1.5);
		\draw[fill=black] (1.75,0.75) circle [radius=0.1];
		\draw (1.25,1.25) node {11};
		\draw[very thick] (1.75,0.75) -- (1.75,0.25) -- (1.25,0.25) -- (1.25,0.75) -- (0.75,0.75) -- (0.75,0.25) -- (0.25,0.25) -- (0.25,1.25) -- (1,1.25);
		\draw[very thick] (1.5,1.25) -- (3.25,1.25) -- (3.25,0.75);
		\draw[very thick] (3.25,0.75) -- (3.25,0.25) -- (2.75,0.25) -- (2.75,0.75) -- (2.25,0.75) -- (2.25,0.25);
		\end{tikzpicture}
		\caption{Two possible paths with $x=11$ in position $J=3$ on a $3 \times 7$ board.}
		\label{fig:x>Jsoln}
	\end{figure}

We have now demonstrated two different solutions with $x \geq 1$ and odd in position $J$ as long as $n-\frac{x-J}{2} \geq 2$.  The largest value of $x$ occurs when we have equality, in which case $2n-4 + J = x$.  Since these constructions pull back through positions 1 through $J-1$, it follows that we can use these same constructions if $x-j$ is placed in position $J-j$ instead, for any $1 \leq j \leq J-1$.  Moreover, since position $J$ is allowed to have odd clues ranging from 1 to $2n-4+J$ it follows that position $J-j$ is allowed to have clues (of appropriate parity) ranging from 1 to $2n-4+J-j$.

The lowest of these bounds is in position 1, where clues range as high as $2n-3$.  Recall that we need this range to reach at least the largest odd number less than or equal to $\frac{3n+1}{2}$.  Observe that if $\frac{3n+1}{2}$ is even (as when $n=5$), then this means we want $2n-3 \geq \frac{3n-1}{2}$ which is true for all $n \geq 5$.  If instead $\frac{3n+1}{2}$ is odd (as when $n=7$), then we need $2n-3 \geq \frac{3n+1}{2}$ which is true for all $n \geq 7$.  Hence, our constructions will successfully cover all cases of clues of appropriate parity up to $\frac{3n+1}{2}$, in any position 1 through $J$, for any $n \geq 5$ as desired.

In principal, this leaves only clues in position $J+1$ if $\frac{n-1}{2}$ happens to be even.  Observe, however, that every one of our solutions with $x$ in position $J$ will also have $x+1$ in position $J+1$.  Hence, we certainly cover all necessary clues in position $J+1$ in this even case as well.  This completes the proof that no single clue can define a $3 \times n$ puzzle for any $n \geq 3$.
\end{proof}

\subsection{Conjectures for larger boards}
As seen in Theorem~\ref{thm:3xn}, the difficulty in proving that our upper bound is minimal arises from the need to rule out the possibility that any arrangement of fewer clues could define a puzzle.  Once again, the lack of a nice way to induct from boards of one size to larger ones is troublesome.  It seems only natural that the minimum number of clues should weakly increase as $m$ increases, but we have failed to discover a proof of this claim.  Despite this, we attempted to gain information about $5 \times n$ and $6 \times n$ puzzles by writing a computer program to create and analyze all solutions in the square cases, $5 \times 5$ and $6 \times 6$.  Our computer program has verified by brute force that two clues is insufficient for either board.  This gives us strong evidence that our stated upper bound on the minimum number of clues in these cases is actually minimal.

For the $7 \times 7$ situation, there are over 27 million different possible solutions.  We have yet to generate and analyze all of those cases, but in particular subsets (for example, solutions with 1 in the middle or on a diagonal just outside the middle), we have yet to find any arrangements of three clues (with 1 as a clue) which define a puzzle.  With these observations in mind, we offer the following conjectures:

\begin{conj}% (i) The minimum number of clues for $3 \times n$ is 2.\\
(i) The minimum number of clues for $5 \times n$ is 3.\\
(ii) The minimum number of clues for $6 \times n$ is 3.\\
(iii) The minimum number of clues for $7 \times n$ is 4.\\
%(v) The minimum number of clues for any $m \times n$ with $m \geq 3$ is exactly $\lceil \frac{m}{2} \rceil$.
\end{conj}

More generally, the fact that our stated upper bound appears to be minimal up through $m=7$ seems to suggest that perhaps it could be minimal for all $m \geq 3$.

\paragraph*{Acknowledgment} 
The authors would like to thank Jonathan Needleman for his suggestion of an approach to constructing $3 \times n$ puzzles with a particular single clue.


\begin{thebibliography}{8}
    \bibitem{CK} Karen L. Collins and Lucia B. Krompart, The number of Hamiltonian paths in a rectangular grid, \textit{Disc. Math.} \textbf{169}, 1997, pp. 29--38.
    \bibitem{CEG} A. R. Conway, I. G. Enting, and A. J. Guttmann, Algebraic techniques for enumerating self-avoiding walks on the square lattice, \textit{J. Phys. A: Math. Gen.} \textbf{26} (1993), pp. 1519--1534.
    \bibitem{Enting} I. G. Enting, Generating functions for enumerating self-avoiding rings on the square lattice, \textit{J. Phys. A: Math. Gen.} \textbf{13} (1980), pp. 3713--3722.
    \bibitem{IPS} Alon Itai, Christos H. Papadimitriou, and Jayme Luiz Szwarcfiter, Hamiltonian paths in grid graphs, \textit{SIAM J. Comput.} \textbf{11} No. 4, Nov 1982, pp. 676--686.
    \bibitem{Jacobsen} Jesper Lykke Jacobsen, Exact enumeration of Hamiltonian circuits, walks and chains in two and three dimensions, \textit{J. of Physics A: Mathematical and Theoretical} \textbf{40} No. 49, Nov 2007.
    \bibitem{MGD} Jean-Marc Mayer, Claude Guez, and Jean Dayantis, Exact computer enumeration of the number of Hamiltonian paths in small plane square lattices, \textit{Physical Review B} \textbf{42} No. 1 (1990).
    \bibitem{SS} Robert Stoyan and Volker Strehl, Enumeration of Hamiltonian circuits in rectangular grids, \textit{J. of Comb. Math. Comb. Comput.} \textbf{21}, 1996, pp. 109--127.
    \bibitem{Thompson} G. L. Thompson, Hamiltonian Tours and Paths in Rectangular Lattice Graphs, \textit{Math. Mag.} \textbf{50} (1977) pp. 147--150.
\end{thebibliography}
\end{document}